\documentclass[12pt]{amsart}

\makeatletter

\usepackage{amssymb,amsmath,amsthm}
\usepackage{color}
\makeindex
\usepackage{hyperref}
\usepackage{enumerate}

\newtheorem{thm}{Theorem}[section]
\newtheorem{theorem}[thm]{Theorem}
\newtheorem{corollary}[thm]{Corollary}
\newtheorem{lemma}[thm]{Lemma}
\newtheorem{proposition}[thm]{Proposition}

\newtheorem{remark}[thm]{Remark}

\theoremstyle{definition}

\theoremstyle{remark}

 \def\bCx{\mathbb C\langle x_1,\dots,x_g\rangle}

\def\e{ e^{i\theta} }
\def\cR{\mathcal R}
 
 \def\cD{\mathcal D}

 \def\cN{\mathcal N}
 \def\matng{M_n(\mathbb C)^g}
 \def\matnh{M_n(\mathbb C)^{\h}}
 \def\cL{\mathcal L}

 \def\h{\tilde{g}}
 \def\math{M(\mathbb C)^{\h}}
 \def\matg{M(\mathbb C)^g}
 \def\cU{\mathcal U}
 \def\cH{\mathcal H}
 \def\cV{\mathcal V}

\def\RR{\mathbb R}
\def\R{\mathbb R}

\def\N{\mathbb N}
\def\C{\mathbb C}

\newcommand{\eps}{\epsilon}
\def\ben{\begin{enumerate} }
\def\een{\end{enumerate} }
\def\beq{\begin{equation} }
\def\eeq{\end{equation} }

\def\tg{\tilde g}

\def\fm{f^{(m)}}

\def\nca{free map }
\def\ncas{free  maps }
\def\ncasp{free  maps. }
\def\NCA{Free Map }
\def\ncap{free  map. }
\def\ncacomma{free  map, }

\def\cncap{analytic free  map. }
\def\cncasp{analytic free  maps. }
\def\cncacomman{analytic free  map, }
\def\cnca{analytic \nca}
\def\cncsma{analytic free self-map }
\def\cncas{analytic \ncas}

\def\fba{bianalytic free map }
\def\fbap{bianalytic free map. }
\def\fbasp{bianalytic free maps. }

\def\fn{f[n]}
\def\fjn{f_j[n]}
\def\fm{f[m]}
\def\fone{f[1]}

\renewcommand{\setminus}{\smallsetminus}

\numberwithin{equation}{section}

\begin{document}

\setcounter{page}{1}
\title{Proper Analytic Free  Maps}

\author[Helton]{J. William Helton${}^1$}
\address{Department of Mathematics\\
  University of California \\
  San Diego}
\email{helton@math.ucsd.edu}
\thanks{${}^1$Research supported by NSF grants
DMS-0700758, DMS-0757212, and the Ford Motor Co.}

\author[Klep]{Igor Klep${}^2$}
\address{Univerza v Ljubljani, Fakulteta za matematiko in fiziko \\
and
Univerza v Mariboru, Fakulteta za naravoslovje in matematiko 
}
\email{igor.klep@fmf.uni-lj.si}
\thanks{${}^2$Research supported by the Slovenian Research Agency grants
P1-0222 and P1-0288.}

\author[McCullough]{Scott McCullough${}^3$}
\address{Department of Mathematics\\
  University of Florida, Gainesville 
   }
   \email{sam@math.ufl.edu}
\thanks{${}^3$Research supported by the NSF grant DMS-0758306.}



\date{9 November 2010}

\begin{abstract} 
This paper concerns {\it analytic free maps}. These 
 maps
are free analogs of classical analytic functions in several
complex variables, and
 are defined in terms of non-commuting variables
amongst which there are no relations - they are 
\emph{free} variables. 
Analytic  free maps
include vector-valued
polynomials in free (non-commuting) variables and
form a canonical class of mappings from
one non-commutative domain $\cD$ in say $g$ variables
to another non-commutative domain $\tilde \cD$ in $\tg$ variables.

As a natural extension of the usual notion,
an analytic free  map is proper if it maps 
the boundary of $\cD$ into the boundary of $\tilde\cD.$
Assuming that both domains contain $0$, we show that
if $f:\cD\to\tilde\cD$ is a proper analytic free map, and $f(0)=0,$
then $f$ is one-to-one.
Moreover, if also $g=\h$,  then
$f$ is invertible and $f^{-1}$ is also an analytic free map.
These conclusions on the map $f$ are the strongest
possible without additional assumptions on the domains
$\cD$ and $\tilde \cD$.
\end{abstract}

\maketitle




\section{Introduction}
 \label{sec:intro}
  The notion of an analytic, free  or non-commutative,
  map arises naturally in free
  probability,
  the study of non-commutative (free) rational functions
  \cite{BGM,Vo1,Vo2,AIM,KVV},
 and
systems theory \cite{HBJP}.

  In this note
  rigidity results for such functions  paralleling those
  for their classical commutative counterparts are established.
   The free setting leads to substantially stronger results.
  Namely, if $f$ is a  proper analytic free  map from
  a non-commutative domain in $g$ variables
  to another in $\tg$ variables, then $f$ is injective
  and $\tg \ge g$. 
  If in addition 
  $\tg=g$, then  $f$  is onto and
  has an inverse which is itself
  a (proper) analytic free map.
 This injectivity conclusion contrasts markedly to the classical
 case where a
 (commutative) \emph{proper} 
  analytic function $f$ from one domain in $\mathbb C^g$
  to another in $\mathbb C^g,$  need not be injective,
  although it must be onto. For  classical theory of 
  some commutative proper analytic maps
  see \cite{DAn}.

 The definitions as used in this paper are given in the
 following section. The main result of the paper is
 in Section \ref{subsec:main-result}. Analytic free analogs
 of classical (commutative) rigidity theorems is the
 theme of Section \ref{sec:analogs}.  The article
 concludes with examples  in Section \ref{sec:Examples},
 all of which involve linear matrix inequalities (LMIs).

\section{Free  Maps}
 \label{sec:maps}
  This section contains the background on
  non-commutative sets and on  {\it free maps}
  at the level of generality needed for this paper.
  As we shall see, free maps which are continuous are also analytic
  in several senses, a fact which (mostly) justifies the 
  terminology  analytic free map in the introduction.
  Indeed one typically thinks of free maps
  as being analytic, but in a weak sense.

  The discussion borrows heavily from
the  recent basic work of
Voiculescu \cite{Vo1, Vo2} and of
Kalyuzhnyi-Verbovetski\u\i{} and Vinnikov
\cite{KVV}, see also the references therein.
These papers contain a power series approach
to free maps
and for more on this one can
see Popescu \cite{Po1,Po2}, or also \cite{HKMS,HKM1}.

\subsection{Non-commutative Sets and Domains}
 \label{subsec:sets}
  Fix a positive integer $g$.
  Given a positive integer $n$, let $\matng$ denote
  $g$-tuples of  $n\times n$ matrices.
  Of course, $\matng$ is naturally identified with
 $M_n(\mathbb C) \otimes\mathbb C^g.$

\renewcommand{\subset}{\subseteq}
   A sequence $\cU=(\cU (n))_{n\in\N},$ where $\cU (n) \subset \matng$,
 is a {\bf non-commutative set} \index{non-commutative set}
 if it is \index{closed with respect to unitary similarity}
 {\bf closed with respect to simultaneous unitary similarity}; i.e.,
 if $X\in \cU (n)$ and $U$ is an $n\times n$ unitary matrix, then
\[
   U^\ast X U =(U^\ast X_1U,\dots, U^\ast X_gU)\in \cU (n);
\]
 and if it is \index{closed with respect to direct sums}
 {\bf closed with respect to direct sums}; i.e.,
  if $X\in \cU (n)$ and $Y\in \cU(m)$ implies
 \[
   X\oplus Y = \begin{pmatrix} X & 0\\ 0 & Y \end{pmatrix}
     \in \cU(n+m).
 \]
 
  Non-commutative sets differ from
  the fully matricial $\C^g$-sets of Voiculescu
 \cite[Section 6]{Vo1} in that the latter are closed
 with respect to simultaneous similarity, not just
    simultaneous \emph{unitary} similarity.
  Remark \ref{rem:sim-vs-unitary} below  briefly discusses  the significance
  of this distinction for the results on proper analytic
  free maps in this paper.

   The non-commutative set $\cU$ is a
  {\bf non-commutative domain} if each
  $\cU (n)$ is open and connected.
  \index{non-commutative domain}
 Of course the sequence $\matg=(\matng)$
 is  itself a non-commutative domain. Given $\varepsilon>0$,
 the set $\cN_\varepsilon = (\cN_\varepsilon(n))$ given by
\beq\label{eq:nbhd}
  \cN_\varepsilon(n)=\big\{X\in\matng : \sum X_j X_j^\ast < \varepsilon^2  \big\}
\eeq
 is a non-commutative domain which we
 call the
  {\bf non-commutative $\varepsilon$-neighborhood of $0$ in $\mathbb C^g$}.
  \index{non-commutative neighborhood of $0$}
  The non-commutative set $\cU$ is {\bf bounded}
  \index{bounded, non-commutative set} if there
  is a $C\in\R$ such that
\beq\label{eq:bd}
   C^2 -\sum X_j X_j^\ast \succ 0
 \eeq
   for every $n$ and $X\in\cU(n)$. Equivalently, for
   some $\lambda\in\RR$, we have $\cU\subset \cN_\lambda$.
  Note that this condition is stronger than asking
 that each $\cU(n)$ is bounded.

  Let $\bCx$ denote the $\C$-algebra freely generated by $g$
non-commuting letters $x=(x_1,\ldots,x_g)$. Its elements are
linear combinations of words in $x$ and are called
{\bf polynomials}.
  Given an $r\times r$ matrix-valued polynomial
  $p\in M_r(\mathbb C) \otimes \bCx$ with
  $p(0)=0$, let $\cD(n)$ denote the connected
  component of
\[
  \{X\in \matng : I+p(X)+p(X)^*  \succ 0\}
\]
  containing the origin.
  The sequence $\cD=(\cD(n))$ is a non-commutative domain
  which is semi-algebraic in nature.
  Note that $\cD$ contains an $\varepsilon>0$ neighborhood
  of $0,$ and that the choice
 \[
   p=  \frac{1}{\varepsilon}
    \begin{pmatrix}\;\; 0_{g\times g}  & \begin{matrix} x_1 \\ \vdots \\x_g\end{matrix} \\
            \;\;  0_{1\times g}  & 0_{1\times 1} \end{pmatrix}
 \]
  gives $\cD = \cN_\varepsilon$.
  Further examples of natural non-commutative domains
  can be generated by considering non-commutative polynomials
  in both the variables $x=(x_1,\dots,x_g)$
  and their formal adjoints, $x^*=(x_1^*,\dots,x_g^*)$.
  The case of domains determined
  by linear matrix inequalities appears in Section \ref{sec:Examples}.

\subsection{Free Mappings}
 \label{subsec:nc-maps}
 Let $\cU$ denote a non-commutative subset of $\matg$
 and let $\h$ be a positive integer.
 A
 {\bf \nca}\index{\nca}
  $f$ from $\cU$ into $\math$ is a sequence
 of functions $\fn:\cU(n) \to\matnh$ which
  {\bf respects intertwining maps}; i.e.,
 if $X\in\cU(n)$, $Y\in\cU(m)$, $\Gamma:\mathbb C^m\to\mathbb C^n$,
  and
 \[
  X\Gamma=(X_1\Gamma,\dots, X_g\Gamma)
   =(\Gamma Y_1,\dots, \Gamma Y_g)=\Gamma Y,
 \]
  then $\fn(X) \Gamma =  \Gamma \fm (Y)$.
 \index{respects intertwining maps}
  Note if $X\in\cU(n)$ it is natural
  to write simply $f(X)$ instead of
 the more cumbersome $\fn(X)$ and
 likewise $f:\cU\to \math$.  In a similar fashion,
 we will often write $f(X)\Gamma=\Gamma f(Y).$

\begin{remark}\rm
 \label{rem:fasvector}
  Each $\fn$ can be represented as
\[
  \fn=\begin{pmatrix} \fn_1 \\ \vdots \\ \fn_{\h} \end{pmatrix}
\]
 where $\fn_j :\cU(n)\to M_n(\mathbb C)$.  Of course, for each
 $j$, the sequence $(\fn_j)$ is a \nca
 $f_j:\cU\to M(\mathbb C)$ with
 $\fjn=\fn_j$.  In particular, if $f:\cU \to \math,$
  $X\in\cU(n)$, and $v=\sum e_j \otimes v_j$, then
\[
  f(X)^* v =\sum f_j(X)^* v_j.
\]
\end{remark}

   Let $\cU$ be a given non-commutative  subset of  $\matg$
   and suppose $f=(\fn)$ is a sequence of functions
  $\fn:\cU(n)\to \matnh$.  The sequence $f$ {\bf respects direct sums}
  if, for each $n,m$ and $X\in\cU(n)$ and  $Y\in \cU(m),$
 \[
   f(X\oplus Y)=f(X) \oplus f(Y).
 \]
   Similarly, $f$
   {\bf respects similarity} if for each $n$ and
  $X,Y\in \cU(n)$ and invertible $n\times n$
  matrix $S$
  such that $XS=SY$,
\[
   f(X) S= Sf(Y).
\]
  The following proposition gives an alternate characterization
  of \ncasp

 \begin{proposition}
  \label{prop:nc-map-alt}
   Suppose $\cU$ is a non-commutative subset of $\matg$.  A sequence
   $f=(\fn)$  of functions $\fn:\cU(n)\to \matnh$
  is a \nca if and only if it
  respects direct sums and similarity.
 \end{proposition}

\begin{proof}
  Observe $f(X)\Gamma=\Gamma f(Y)$ if and only if
 \[
    \begin{pmatrix} f(X) & 0 \\ 0  & f(Y) \end{pmatrix}
    \begin{pmatrix} I & \Gamma \\ 0 & I \end{pmatrix}
   = \begin{pmatrix} I & \Gamma \\ 0 & I \end{pmatrix}
    \begin{pmatrix} f(X) & 0 \\ 0  & f(Y) \end{pmatrix}.
 \]
  Thus if $f$ respects direct sums and similarity, then
 $f$ respects intertwining.

  On the other hand, if $f$ respects intertwining then,
 by choosing $\Gamma$ to be an appropriate projection,
 it is easily seen that
  $f$ respects direct sums too.
\end{proof}

\begin{remark}\rm
 \label{rem:sim-vs-unitary}
   Let $\cU$ be a non-commutative domain in $M(\C)^g$ and
  suppose $f:\cU\to M(\C)^{\tg}$ is a free map. If $X\in \cU$
  is similar to $Y$ with $Y=S^{-1}X S$, then we can 
  define $f(Y) = S^{-1}f(X)S$. In this way $f$ naturally
 extends to a free map on $\cH(\cU)\subset M(\C)^g$ defined by
\[
  \cH(\cU)(n)=\{Y\in M_n(\C)^g: \text{ there is an } X\in \cU(n)
     \text{ such that $Y$ is similar to $X$}\}.
\]
  Thus if $\cU$ is  a domain of holomorphy, then $\cH(\cU)=\cU$.

 On the other hand, because our results on proper analytic free maps
  to come depend strongly upon 
  the  non-commutative set $\cU$ itself,
   the distinction between non-commutative sets and
   fully matricial sets as in  \cite{Vo1} is important.
    See also  \cite{HM,HKM2}.
\end{remark}

 We close this subsection with the following simple observation.

\begin{proposition}
 \label{prop:range}
  If $\cU$ is a non-commutative subset of $\matg$ and
  $f:\cU\to \math$ is a \ncacomma then the range of $f$, equal to
  the sequence $f(\cU)=\big(f(\cU(n))\big)$, is itself
  a non-commutative subset of $\math$.
\end{proposition}

\subsection{A Continuous \NCA is  Analytic}
  Let $\cU\subset \matg$ be a non-commutative set.
  A \nca   $f:\cU\to \math$ is {\bf continuous} if each
  $\fn:\cU(n)\to \matnh$ is continuous. \index{continuous}
  Likewise,
  if $\cU$ is a non-commutative domain, then
  $f$ is called {\bf analytic} if each $\fn$ is analytic.
 \index{analytic}
This implies the existence of directional derivatives for all directions
 at each point in the domain, and this is the property we shall use later
below.

\begin{proposition}
 \label{prop:continuous-analytic}
  Suppose $\cU$ is a non-commutative domain in $\matg$.
\ben[\rm (1)]
\item
  A continuous \nca $f:\cU\to \math$ is analytic.
\item If $X\in\cU(n)$, and $H\in\matng$ has sufficiently small norm,
then
 \[
  f\begin{pmatrix} X & H \\ 0 & X\end{pmatrix}
   = \begin{pmatrix} f(X) & f^\prime(X)[H] \\ 0 & f(X)\end{pmatrix}.
 \]
\een
\end{proposition}

 The proof invokes the following lemma which also plays
 an important role in the next subsection.

\begin{lemma}
 \label{lem:2x2}
  Suppose $\cU\subset \matg$ is a non-commutative set
 and $f:\cU\to \math$ is a \ncap
  Suppose $X\in\cU(n)$, $Y\in \cU(m)$, and
  $\Gamma$ is an $n\times m$ matrix. Let
\beq\label{eq:2x21}
    C_j = X_j\Gamma -\Gamma Y_j, \quad
    Z_j = \begin{pmatrix} X_j & C_j \\ 0 & Y_j \end{pmatrix}.
 \eeq
  If
  $Z=(Z_1,\dots,Z_g)\in \cU(n+m)$, then
\beq\label{eq:2x22}
  f_j(Z)=
\begin{pmatrix}
    f_j(X) & f_j(X)\Gamma -\Gamma f_j(Y) \\
                        0 & f_j(Y)
\end{pmatrix}
\eeq
\end{lemma}

This formula generalizes to larger block matrices.

\begin{proof}
  With
 \[
    S=\begin{pmatrix} I & \Gamma \\ 0 & I \end{pmatrix}
 \]
  we have
 \[
   \tilde{Z}_j= \begin{pmatrix} X_j & 0 \\ 0 & Y_j \end{pmatrix}
        = S Z_j S^{-1}.
 \]
  Thus, writing $f=(f_1,\dots,f_{\h})^T$ and using
  the fact that $f$ respects intertwining maps,  for each $j$,
 \[
    f_j(Z) =  S f_j(\tilde{Z}) S^{-1}
         =  \begin{pmatrix} f_j(X) & f_j(X)\Gamma -\Gamma f_j(Y) \\
                        0 & f_j(Y) \end{pmatrix}.\qedhere
 \]
\end{proof}

\begin{proof}[Proof of Proposition {\rm\ref{prop:continuous-analytic}}]
 Fix $n$ and $X\in \cU(n)$.
 Because $\cU(2n)$ is open and $X\oplus X\in \cU(2n)$,
for every $H\in \matng$ of sufficiently small norm
 the tuple with $j$-th entry
\[
  \begin{pmatrix} X_j & H_j \\ 0 & X_j \end{pmatrix}
\]
 is in $\cU(2n)$. Hence,
 for $z\in\mathbb C$ of small modulus, the tuple $Z(z)$ with
 $j$-th entry
\[
  Z_j(z)=\begin{pmatrix} X_j+zH_j & H_j \\ 0 & X_j  \end{pmatrix}
\]
 is in $\cU(2n)$.  Note that the choice (when $z\ne 0$) of
 $\Gamma(z)=\frac{1}{z}$,
 $X=X+zH$ and $Y=X$
  in Lemma \ref{lem:2x2}
 gives this $Z(z)$. Hence,  by Lemma \ref{lem:2x2},
\[
  f(Z(z))= \begin{pmatrix} f(X+zH) & \frac{f(X+zH)-f(X)}{z} \\ 0 & f(X)
      \end{pmatrix}.
\]
 Since $Z(z)$ converges as $z$ tends to $0$ and $f[2n]$
 is assumed continuous, the limit
\[
  \lim_{z\to 0} \frac{f(X+zH)-f(X)}{z}
\]
 exists. This proves that $f$ is analytic at $X$.  It also
 establishes the moreover portion of the proposition.
\end{proof}

\begin{remark}\rm
  Kalyuzhnyi-Verbovetski\u\i{} and Vinnikov \cite{KVV} are developing
 general results based on very weak hypotheses with the conclusion that
 $f$ is (in our language) an \cncap
 Here we will assume
 continuity whenever expedient.

 For perspective we mention power series.
 It is shown in \cite[Section 13]{Vo2}
 that an \cnca
 $f$ has a formal power series expansion
 in the non-commuting variables, which indeed is a powerful way
 to think of \cncasp  Voiculescu also gives elegant
 formulas for the coefficients of the power series expansion of $f$ in terms
 of clever evaluations of $f$. Convergence properties for bounded
\cncas are studied in \cite[Sections 14-16]{Vo2}; see also
\cite[Section 17]{Vo2} for a bad unbounded example. We do not dwell
 on this since power series are not essential to this paper.
\end{remark}

\section{A Proper \NCA is Bianalytic Free}
 \label{subsec:main-result}
   Given non-commutative domains $\cU$ and $\cV$ in
  $\matg$ and $\math$ respectively, a
  \nca $f:\cU\to\cV$ is {\bf proper} \index{proper}
  if each $\fn:\cU(n)\to \cV(n)$ is proper in the
  sense that if $K\subset \cV(n)$ is compact, then
  $f^{-1}(K)$ is compact.
  In particular,  for all $n$,
  if $(z_j)$ is a sequence in $\cU(n)$
  and $z_j\to\partial\cU(n)$, then
  $f(z_j)\to\partial\cV(n)$.
 In the case $g=\h$ and both $f$ and $f^{-1}$ are (proper)
 \cncas
 we say $f$ is 
 a 
{\bf bianalytic} free map.
  The following theorem is a central result of this paper.

\begin{theorem}
\label{thm:oneone}
  Let $\cU$ and $\cV$ be non-commutative domains containing $0$
  in $\matg$ and $\math$, respectively and
  suppose $f:\cU\to \cV$ is a \ncap
\begin{enumerate}[\rm (1)]
\item\label{it:1to1}
If $f$ is proper, then it is one-to-one,
  and $f^{-1}:f(\cU)\to \cU$ is a \ncap
\item\label{it:1to1ugly}
  If, for each $n$ and $Z\in\matnh$,
  the set $\fn^{-1}(\{Z\})$ has compact closure in $\cU$, 
  then $f$ is one-to-one
  and moreover, $f^{-1}:f(\cU)\to \cU$ is a \ncap
\item\label{it:xto1}
If $g=\h$ and $f:\cU\to\cV$ is
 proper and  continuous, then $f$ is bianalytic.
\end{enumerate}
\end{theorem}

\begin{corollary}
 \label{cor:bianalytic-free}
  Suppose $\cU$ and $\cV$ are non-commutative domains in
 $\matg$. If $f:\cU\to \cV$ 
  is a free map and if each $f[n]$ is bianalytic, 
  then $f$ is a bianalytic free map.
\end{corollary}

\begin{proof}
 Since each $\fn$ is bianalytic, each $\fn$ is proper. Thus $f$ is
 proper. Since also $f$ is a \ncacomma by Theorem
  \ref{thm:oneone}\eqref{it:xto1}  $f$ is
 a \fbap
\end{proof}

 Before proving Theorem \ref{thm:oneone} we establish the following
 preliminary result which is of independent interest and whose proof
 uses the full strength of Lemma \ref{lem:2x2}.

\begin{proposition}
\label{prop:oneone}
  Let $\cU\subset \matg$ be a non-commutative domain
  and suppose $f:\cU\to \math$ is a \ncap
   Suppose further that $X\in\cU(n)$, $Y\in \cU(m)$, $\Gamma$ is an
  $n\times m$ matrix, and
 \[
   f(X)\Gamma = \Gamma f(Y).
 \]
   If
   $f^{-1}\big(\{f(X)\oplus f(Y)\}\big)$ has compact closure
   in $\cU$,
   then $X\Gamma = \Gamma Y.$
\end{proposition}

\begin{proof}
  As in Lemma \ref{lem:2x2},
  let $C_j = X_j\Gamma  -\Gamma Y_j$.
  For $0<t$ sufficiently small, $Z(t)\in \cU(n+m)$, where
 \beq\label{eq:propcomp}
   Z_j(t) =\begin{pmatrix} X_j & t C_j \\ 0 & Y_j \end{pmatrix}.
 \eeq
  If $f(X)\Gamma = \Gamma f(Y),$ then, by Lemma
  \ref{lem:2x2},
  \[
     f_j(Z(t))  = \begin{pmatrix} f_j(X)
                        & t\big( f_j(X)\Gamma -\Gamma f_j(Y) \big)\\
                     0 & f_j(Y) \end{pmatrix} \\
       =  \begin{pmatrix} f_j(X) & 0 \\ 0 & f_j(Y) \end{pmatrix}.
 \]
  Thus, $f_j(Z(t))=f_j(Z(0)).$ In particular,
 \[
   f^{-1}\big(\{f(Z(0))\}\big) \supseteq \{Z(t): t\in \mathbb C\}\cap \cU.
 \]
  Since this set has, by assumption, compact closure in $\cU$,
 it follows that $C=0$; i.e.,
  $X\Gamma=\Gamma Y$.
\end{proof}

We are now ready to prove that
 a proper \nca is one-to-one and even a \fba if
 continuous and  mapping
between domains of the same dimension.

\begin{proof}[Proof of Theorem {\rm\ref{thm:oneone}}]
If $f$ is proper, then $f^{-1}(\{Z\})$ 
has compact closure in $\cU$
for every $Z\in\math$. Hence \eqref{it:1to1} is a consequence of
\eqref{it:1to1ugly}.

   For
  \eqref{it:1to1ugly}, invoke Proposition \ref{prop:oneone}
  with $\Gamma=\gamma I$ to conclude that $f$ is injective.
  Thus $f:\cU\to f(\cU)$ is a bijection from one non-commutative
 set to another.  Given $W,Z\in f(\cU)$ there exists
 $X,Y\in\cU$ such that $f(X)=W$ and $f(Y)=Z$. If
  moreover, $W\Gamma =\Gamma Z$, then $f(X) \Gamma =\Gamma f(Y)$
 and Proposition \ref{prop:oneone} implies $X\Gamma=\Gamma Y$; i.e.,
  $f^{-1}(W)\Gamma=\Gamma f^{-1}(Z)$. Hence $f^{-1}$ is itself
  a \ncap

 Let us now consider
 \eqref{it:xto1}.
  Using the continuity hypothesis and Proposition
  \ref{prop:continuous-analytic},
  for each $n$, the map $\fn:\cU(n)\to \cV(n)$
  is analytic. By hypothesis each $\fn$ is also proper and
  hence its  range
 is $\cV(n)$ by \cite[Theorem 15.1.5]{Rudin}.

  Now $\fn:\cU(n)\to \cV(n)$ is one-to-one, onto and analytic, so
  its inverse is analytic.
Further, by
  the already proved part of the
  theorem, $f^{-1}$ is an \cncap
\end{proof}

For both completeness and later use we record the following
 companion to Lemma \ref{lem:2x2}.

\begin{proposition}
 \label{prop:fprime}
  Let $\cU\subset\matg$ and $\cV\subset\math$ be  non-commutative
  domains.
  If $f:\cU\to\cV$ is a
  proper \cnca
 and if $X\in\cU(n)$, then
  $f^\prime(X):\matng\to \matnh$ is one-to-one.
  In particular, if $g=\h$, then $f^\prime(X)$
  is a vector space isomorphism.
\end{proposition}

\begin{proof}
  Suppose $f^\prime(X)[H]=0$.
 We scale $H$ so that $\begin{pmatrix} X & H \\ 0 & X\end{pmatrix} \in\cU$.
 From Proposition \ref{prop:continuous-analytic}
\[
  f\begin{pmatrix} X & H \\ 0 & X\end{pmatrix}
   = \begin{pmatrix} f(X) & f^\prime(X)[H] \\ 0 & f(X)\end{pmatrix}=
  \begin{pmatrix} f(X) & 0 \\ 0 & f(X)\end{pmatrix}
   = f\begin{pmatrix} X & 0 \\ 0 & X\end{pmatrix}.
\]
By the injectivity of $f$ established in Theorem \ref{thm:oneone},
$H=0$.
\end{proof}

\subsection{The Main Result is Sharp}
Key to the proof of Theorem \ref{thm:oneone} is
testing $f$ on the special class of matrices of the
form \eqref{eq:propcomp}. One naturally asks if
the hypotheses of the theorem in fact yield stronger
conclusions, say by plugging in richer classes of test matrices.
The answer to this question is no:
suppose $f$ is any \cnca
from $g$ to $g$ variables defined on a 
neighborhood $\cN_\eps$ of $0$ with $f(0)=0$ and $\fone'(0)$
invertible. 
Under mild additional assumptions (e.g.~the lowest 
eigenvalue of $f'(X)$ or the norm $\|f'(X)\|$
is bounded away from $0$ for $X\in\cN_\eps(n)$ independently of the
size $n$)
then there are non-commutative domains $\cU$ and $\cV$ with
$f:\cU\to\cV$ meeting the hypotheses of the theorem.

Indeed, consider (for fixed $n$)
the analytic function $\fn$ on $\cN_\eps(n)$.
Its derivative at $0$ is invertible; in fact,
$\fn'(0)$ is unitarily equivalent to
$I_n\otimes \fone'(0)$, cf.~Lemma \ref{lem:ampliate} below.
By the implicit function theorem,
there is a small $\delta$-neighborhood of $0$ on which
$\fn^{-1}$ is defined and analytic. By our assumptions and the bounds on the
size of this neighborhood given in \cite{Wan}, $\delta>0$
may be chosen to be independent of $n$.
This gives rise to a non-commutative
domain $\cV$ and the \cnca
$f^{-1}:\cV\to\cU$, where $\cU=f^{-1}(\cV)$.
Note $\cU$ is open (and hence a non-commutative domain)
since $f^{-1}(n)$ is analytic and one-to-one.
It is now clear
that $f:\cU\to\cV$ satisfies the hypotheses of Theorem \ref{thm:oneone}.

We just saw that absent more conditions on the non-commutative
domains $\cD$ and $\tilde \cD$, nothing beyond bianalytic free
can be concluded about $f$.
The authors, for reasons not gone into here, are particularly
interested in convex domains, the paradigm being
those given by what are called LMIs. These will be discussed
in Section \ref{sec:Examples}.
Whether or not convexity of the domain or range of an analytic
free  $f$
has a highly restrictive impact on
$f$ is a serious open question.

\section{Several Analogs to Classical Theorems}
\label{sec:analogs}
The conclusion of Theorem \ref{thm:oneone} is sufficiently
strong that most would say that it does not have a classical analog.
In this section \cnca analogs of classical
several complex variable theorems are obtained by combining
the corresponding classical theorem and Theorem
\ref{thm:oneone}. Indeed,
hypotheses for these analytic \nca results are weaker than their
classical analogs would suggest.

\subsection{A Free Caratheodory-Cartan-Kaup-Wu (CCKW) Theorem}
\label{sec:onto2}
 The commutative Caratheodory-Cartan-Kaup-Wu (CCKW) Theorem
 \cite[Theorem 11.3.1]{Krantz} says that
 if $f$ is an analytic self-map of a bounded domain
 in $\mathbb C^g$ which fixes
 a point $P$, then the eigenvalues of $f^\prime(P)$
 have modulus at most one. Conversely, if the eigenvalues
 all have  modulus one, then $f$ is in fact an automorphism;  and
 further if $f^\prime(P)=I$, then $f$ is the identity.
 The CCKW Theorem together with Corollary \ref{cor:bianalytic-free}
 yields Corollary \ref{cor:cckw1} below.
 We note
 that Theorem \ref{thm:oneone} can also be thought of
 as a non-commutative CCKW theorem in that
it concludes, like the CCKW Theorem does,  
 that a  map $f$ is bianalytic, 
 but under the (rather different) assumption that $f$ is proper.

\begin{corollary}
\label{cor:cckw1}
 Let $\cD$ be a given bounded
non-commutative domain
 which contains $0$.
 Suppose
 $f:\cD \to \cD$ is
 an \cncap
   Let $\phi$ denote the
 mapping $\fone:\cD(1)\to\cD(1)$ and assume $\phi(0)=0.$
\begin{enumerate}[\rm (1)]
 \item  If  all the eigenvalues of
 $\phi^\prime(0)$ have modulus one,
   then $f$ is a  bianalytic free map; and
 \item  if $\phi^\prime(0)=I$, then $f$ is the identity.
\end{enumerate}
\end{corollary}

 The proof uses the following lemma,  whose proof is trivial if it is
 assumed that $f$ is continuous (and hence analytic) and then
one works with the formal power series representation for a 
free analytic function.

\begin{lemma}
 \label{lem:ampliate}
 Keep the notation and hypothesis of Corollary {\rm\ref{cor:cckw1}}.
 If $n$ is a positive integer and  $\Phi$
 denotes the mapping $\fn:\cD(n)\to\cD(n)$, then
 $\Phi^\prime(0)$ is unitarily equivalent to
 $I_n\otimes \phi^\prime(0)$.
\end{lemma}

\begin{proof}
 Let $E_{i,j}$ denote the matrix units for
 $M_n(\C)$.  Fix $h\in\mathbb C^g$. Arguing as in the proof of
 Proposition \ref{prop:fprime} gives, for $k\ne \ell$
 and $z\in \mathbb C$ of small modulus,
\[
  \Phi((E_{k,k}+E_{k,\ell})\otimes z h)
   = (E_{k,k} +E_{k,\ell})\otimes \phi(zh).
\]
 It follows that
\[
  \Phi^\prime(0)[(E_{k,k}+E_{k,\ell})\otimes h]
    = (E_{k,k}+E_{k,\ell})\phi^\prime(0)[h].
\]
 On the other hand,
\[
  \Phi^\prime(0)[E_{k,k}\otimes h] = E_{k,k}\otimes \phi^\prime(0)[h].
\]
 By linearity of $\Phi^\prime(0)$, it follows that
\[
 \Phi^\prime(0)[E_{k,\ell}\otimes h] = E_{k,\ell}\otimes \phi^\prime(0)[h].
\]
 Thus, $\Phi^\prime(0)$ is  unitarily equivalent to
 $I_{n}\otimes \phi^\prime(0).$
\end{proof}

\begin{proof}[Proof of Corollary {\rm\ref{cor:cckw1}}]
 The hypothesis  that $\phi^\prime(0)$ has eigenvalues
 of modulus one, implies, by Lemma \ref{lem:ampliate},
 that, for each $n$, the eigenvalues
 of $\fn^\prime(0)$ all have modulus one. Thus,
 by the CCKW Theorem, each $\fn$ is an automorphism.
 Now Corollary \ref{cor:bianalytic-free} implies $f$
 is a \fbap

 Similarly, if $\phi^\prime(0)=I_g$, then $\fn^\prime(0)=I_{ng}$
 for each $n$. Hence,  by the CCKW Theorem, $\fn$
 is the identity for every $n$ and
 therefore $f$ is itself the identity.
\end{proof}

Note a classical bianalytic function $f$ is
completely determined by its value and differential
at a point
 (cf.~a remark after Theorem 11.3.1 in \cite{Krantz}).
Much the same is true for \cncas and for the same reason.

\begin{proposition}
 \label{prop:derivative0}
   Suppose $\cU,\cV\subset \matg$ are non-commutative domains, $\cU$
   is bounded, both contain $0,$
 and $f,g:\cU\to \cV$ are proper \cncasp
 If $f(0)=g(0)$ and $f^\prime(0)=g^\prime(0)$, then $f=g$.
\end{proposition}

\begin{proof}
 By Theorem \ref{thm:oneone} both $f$ and $g$ are
 \fbasp
Thus $h=f\circ g^{-1}:\cU\to \cU$
 is a \fba fixing $0$ with $h[1]^\prime(0)=I$.
 Thus, by Corollary \ref{cor:cckw1}, $h$ is the identity.
 Consequently $f=g$.
\end{proof}

\subsection{Circular Domains}
  A subset $S$ of a complex vector space is
  {\bf circular} if $\exp(it) s\in S$  whenever
  $s\in S$ and $t\in\R$.
  A non-commutative domain $\cU$ is circular
  if each $\cU(n)$ is circular. \index{circular domain}

 Compare the following theorem  to
 its commutative counterpart \cite[Theorem  11.1.2]{Krantz}
 where the domains $\cU$ and $\cV$ are the same.

\begin{theorem}
\label{thm:circLin}
  Let $\cU$ and $\cV$ be bounded non-commutative domains
  in $\matg$ and $\math$, respectively, both
  of which contain $0$.
  Suppose $f:\cU\to \cV$ is a
  proper \cnca
  with $f(0)=0$.
  If $\cU$ and  the range
 $\cR:= f(\cU)$ of $f$
 are circular,  then $f$ is linear.
\end{theorem}

 The domain $\cU=(\cU(n))$ is {\bf convex}
if each $\cU(n)$ is a convex set.

\begin{corollary}
  Let $\cU$ and $\cV$ be bounded non-commutative domains
  in $\matg$  both
  of which contain $0$.
  Suppose $f:\cU\to \cV$ is a
  proper \cnca  with $f(0)=0$.
  If both $\cU$ and $\cV$ are circular and if one is convex,
  then so is the other.
\end{corollary}

This corollary is an immediate
consequence of Theorem \ref{thm:circLin} and the fact
(see Theorem \ref{thm:oneone}\eqref{it:xto1}) that
$f$ is onto $\cV$. 

 We admit the hypothesis that the range $\cR= f(\cU)$ of $f$
 in Theorem \ref{thm:circLin}
 is circular seems pretty contrived when the domains
 $\cU$ and $\cV$ have a different number of variables.
 On the other hand if they  have the same number of variables
 it is the same as $\cV$ being circular since by
 Theorem \ref{thm:oneone}, $f$ is onto.

\begin{proof}[Proof of Theorem {\rm\ref{thm:circLin}}]
 Because $f$ is a proper \nca it is injective
 and its inverse (defined on $\cR$) is a \nca
 by Theorem \ref{thm:oneone}.  Moreover, using the
 analyticity of $f$,  its derivative
 is pointwise injective by Proposition \ref{prop:fprime}.
 It follows that each $\fn:\cU(n)\to \matnh$
 is an embedding \cite[p.~17]{GP}.
 Thus, each $\fn$  is a homeomorphism onto its range and
 its  inverse $\fn^{-1}=f^{-1}[n]$ is continuous.

 Define $F: \cU \to \cU$ by
\beq\label{eq:defF}
  F(x):=  f^{-1}\big( e^{- i\theta} f(\e x) \big)
\eeq
 This function respects direct sums and similarities,
 since it is the composition of maps which do.
  Moreover, it is continuous by the discussion above.
 Thus $F$ is an \cncap

 Using the relation $\e f(F(x))= f(\e)$
 we find $\e f^\prime(F(0))F^\prime(0)=f^\prime(0)$.
 Since $f^\prime(0)$ is injective, $\e F^\prime(0)=I.$
 It follows from Corollary \ref{cor:cckw1}(2)
 that $F(x)=\e x$ and thus,
 by \eqref{eq:defF},
 $f(\e x)=\e f(x)$. Since this holds for every
  $\theta$, it follows that $f$ is linear.
\end{proof}

If $f$ is not assumed to map $0$ to $0$ (but instead fixes
some other point), then a
proper self-map 
need not be linear.
This follows from the  example  we discuss in
Section \ref{sec:exYes}.

\begin{remark}\rm
 A consequence of the Kaup-Upmeier series of papers
 \cite{BKU,KU} shows that
 given
 two bianalytically equivalent bounded circular 
 domains in $\C^g$, there
 is a \emph{linear} bianalytic map between them.
 We believe this result   extends to the present non-commutative setting.g
\end{remark}

\section{Maps in One Variable, Examples}
 \label{sec:Examples}
  This section contains two examples.  The first shows that
  the circled hypothesis is needed in Theorem \ref{thm:circLin}.
Our second example concerns $\cD$,
a non-commutative domain in one variable
  containing the origin, and $b:\cD\to \cD$ a proper
  \cnca with $b(0)=0$. It follows that $b$ is bianalytic
  and hence $b[1]^\prime(0)$
  has modulus one. Our second example shows that this setting can force
  further restrictions on $b[1]^\prime(0)$.
  The non-commutative domains of both examples are LMI domains; i.e.,
  they are the non-commutative  solution set of a linear
  matrix inequality (LMI).  Such domains are convex, and play a major
  role in the important area of semidefinite programming; see \cite{WSV}
or the excellent survey \cite{Nem}.

\subsection{LMI Domains}
  A special case of the non-commutative domains
  are those described by a linear matrix inequality. 
  Given a positive integer $d$
  and  $A_1,\dots,A_g \in M_d(\C),$ the
  linear matrix-valued polynomial
\[
  L(x)=\sum A_j x_j\in M_d(\C)\otimes \bCx
\]
  is a {\bf truly linear pencil}. \index{truly linear pencil}
  Its adjoint is, by definition,
$
  L(x)^*=\sum A_j^* x_j^*.
$
  Let
\[
  \cL(x) = I_d + L(x) +L(x)^*.
\]
  If $X\in \matng$, then $\cL(X)$ is defined by the canonical substitution,
\[
  \cL(X) = I_d\otimes I_n
          +\sum A_j\otimes X_j +\sum A_j^* \otimes X_j^*,
\]
 and yields a symmetric $dn\times dn$ matrix.
  The inequality $\cL(X)\succ 0$ for tuples $X\in\matg$
  is a {\bf linear matrix inequality (LMI)}. \index{LMI}
  \index{linear matrix inequality} The sequence of solution sets
 $\cD_{\cL}$ defined by
\[
 \cD_{\cL}(n) = \{X\in\matng : \cL(X)\succ 0\}
\]
 is a non-commutative domain which contains a neighborhood
 of $0$. It is called a {\bf non-commutative (NC) LMI domain}.

\subsection{A Concrete Example of a Nonlinear Bianalytic Self-map
on an NC LMI Domain}
\label{sec:exYes}
It is surprisingly difficulty to find proper self-maps on LMI domains
which are not linear. This section contains the only such example,
up to trivial modification, of which
we are aware.   Of course, by Theorem \ref{thm:circLin} the underlying
domain cannot be circular.

In this example the domain is a one-variable LMI domain.
Let $$A=\begin{pmatrix} 1&1\\ 0&0\end{pmatrix}$$
and let $\cL$ denote the univariate $2\times 2$ linear pencil,
$$\cL(x):= I + Ax + A^* x^*
=
\begin{pmatrix} 1 + x +x^* & x \\
                        x^* & 1 \end{pmatrix}.
$$
Then 
$$\cD_\cL=\{X\mid \| X-1 \| < \sqrt 2\}.$$
To see this note $\cL(X) \succ 0$
if and only if $1+X+X^*-XX^*\succ0$,
which is in turn equivalent to $(1-X)(1-X)^*\prec 2$.

\begin{proposition}\label{prop:ex}
For real $\theta$, consider
$$
f_\theta (x):= \frac{ e^{i \theta} x}{1+x- e^{i \theta} x}.
$$
\ben[\rm (1)]
\item
$f_\theta:\cD_\cL\to\cD_\cL$ is a proper
\cncacomman
$f_\theta(0)=0$, and $f^\prime_\theta(0)=\exp(i \theta)$.
\item
Every proper \cnca  $f:\cD_\cL\to\cD_\cL$ fixing the origin
equals one of the $f_\theta$.
\een
\end{proposition}

\begin{proof}
Item (1) follows from a straightforward computation:
\begin{gather*}
(1-f_\theta(X))(1-f_\theta(X))^* \prec 2 \iff
 \left(1-\frac{ e^{i \theta} X}{1+X- e^{i \theta} X}\right)
\left(1- \frac{ e^{i \theta} X}{1+X- e^{i \theta} X}\right)^* \prec 2 \\
\iff
\left(
\frac {1+X-2 e^{i \theta} X}{1+X- e^{i \theta} X} \right) \left(
\frac{1+X-2 e^{i \theta} X}{1+X- e^{i \theta} X}
\right)^* \prec 2 \\ \iff
\left(
{1+X-2 e^{i \theta} X}\right)
\left(
{1+X-2 e^{i \theta} X} \right)^*
\prec 2
\left(
{1+X- e^{i \theta} X}\right)
\left(
{1+X- e^{i \theta} X} \right)^* \\ \iff
1+X+X^*-XX^*\succ 0 \iff
(1-X)(1-X)^*\prec 2.
\end{gather*}

Statement (2) follows from the uniqueness of a bianalytic map
carrying $0$ to $0$
with a prescribed derivative.
\end{proof}

\subsection{Example of Nonexistence of a Bianalytic Self-map on an NC LMI Domain}
\label{sec:exNo}
Recall that
a bianalytic $f$ with $f(0)=0$ is
completely determined by its differential
at a point.
Clearly, when $f^\prime(0)=1$, then $f(x)=x$.
Does a proper \cncsma exist for each $f^\prime(0)$
of modulus one?  In the previous example this was the case.
For the domain in the example in this subsection, again in one variable,
there is no proper  \cncsma  whose derivative
at the origin is $i$.

The domain will be a ``non-commutative ellipse'' described as $\cD_\cL$
with $\cL(x):= I + A x + A^* x^*$
for $A$ of the form
$$A :=
\begin{pmatrix} C_1 & C_2\\ 0 & -C_1\end{pmatrix},
$$
where $C_1, C_2\in\R$.
There is a choice of parameters
in $\cL$ such that
there is no proper \cncsma
$b$ on $\cD_\cL$ with $b(0)=0$,
and $b'(0) = i$.

Suppose
$b:\cD_\cL\to\cD_\cL$ is a proper \cncsma
with $b(0)=0$,
and $b'(0) = i$. By Theorem \ref{thm:oneone}, $b$ is bianalytic.
In particular, $b[1]:\cD_\cL(1)\to\cD_\cL(1)$ is bianalytic. By the Riemann mapping
theorem there is a conformal map $f$ of the unit disk onto $\cD_\cL(1)$
satisfying $f(0)=0$.
Then
\beq\label{eq:b1}
b[1](z)= f \big( i f^{-1}(z)\big).
\eeq
(Note that $b[1] \circ b[1] \circ b[1] \circ b[1]$ is the identity.)

To give an explicit example, we recall some
special functions involving elliptic integrals.
Let $K(z,t)$ and $K(t)$ be the normal and complete elliptic integrals
of the first kind, respectively, that is,
$$
K(z,t)= \int_0^z \frac {dx}{\sqrt{(1-x^2)(1-t^2x^2)}},
\quad
K(t)=K(1,t).
$$
Furthermore, let
$$
\mu(t)=\frac \pi2 \frac {K(\sqrt{1-t^2})}{K(t)}.
$$

Choose the axis for the non-commutative ellipse as follows:
$$
a= \cosh \left( \frac12 \mu\big(\frac 23\big)\right),
\quad
b=\sinh  \left( \frac12 \mu\big(\frac 23\big)\right).
$$
Then
$$
C_1=\frac12 \sqrt{ \frac 1{a^2}-\frac1{b^2}}, \quad
C_2=\frac1b.
$$
The desired conformal mapping is \cite{Scw,Sze}
$$
f(z)=\sin \left(
\frac{\pi}{2 K(\frac23)} K \Big( \frac z{\sqrt \frac23},\frac23\Big) \right).
$$
Hence $b[1]$ in \eqref{eq:b1}
can be explicitly computed
(for details see the Mathematica notebook {\tt Example53.nb}
available under {\it Preprints} on
\url{http://srag.fmf.uni-lj.si}). It
has a power series expansion
\beq\label{eq:b12}
\begin{split}
b[1](z) & =
i z-\frac{1}{27} i \left(9-\frac{52
   K\left(\frac{4}{9}\right)^2}{\pi ^2}\right)
   z^3+i\frac{ \left(9 \pi ^2-52
   K\left(\frac{4}{9}\right)^2\right)^2 }{486
   \pi ^4}z^5+ O(z^7) \\
& \approx i\, (1+ 0.30572  z^3 + 0.140197  z^5).
\end{split}
\eeq

This power series expansion has a radius of convergence
$\geq\eps>0$
and thus induces an analytic free mapping
 $\cN_\eps\to M(\C)$. 
By analytic continuation, this function coincides with $b$.
This enables us to evaluate $b(zN)$ for a nilpotent $N$.

Let $N$ be an order $3$ nilpotent,
$$
N=\begin{pmatrix}
0 & 1 & 0 & 0\\
0 & 0 & 1 & 0\\
0 & 0 & 0 & 1\\
0 & 0 & 0 & 0\\
\end{pmatrix}.
$$
Then $r\in\R$ satisfies $rN\in\cD_\cL$ if and only if
$-1.00033 \leq r \leq 1.00033=:r_0$.
(This has been computed symbolically in the exact arithmetic using Mathematica, and
the bounds given here are just approximations.)
 However,
$b(r_0N)\in\cD_\cL\setminus\partial\cD_\cL$
contradicting the properness.
(This was established by computing
the $8\times 8$ matrix $\cL\big(b(r_0N)\big)$
symbolically thus ensuring it is exact.
Then we apply a numerical eigenvalue solver to see that it is
positive definite with smallest eigenvalue $0.0114903\ldots$.)
We conclude that the proper \cncsma
 $b$
does not exist.

\end{document}